\documentclass[12pt]{amsart}
\usepackage{graphicx}
\title[Completeness of trajectories]{Completeness of trajectories associated to Appell hypergeometric functions}
\author[L. Boulton]{Lyonell Boulton$^\beta$} 
\address{$^\beta$Department of Mathematics and Maxwell Institute for Mathematical Sciences, Heriot-Watt University, Edinburgh, EH14 4AS.}
\email{L.Boulton@hw.ac.uk}
\date{12th April 2023}
\newcommand{\sgn}{\operatorname{sgn}}
\newcommand{\sn}{\operatorname{sn}}

\newtheorem{Lemma}{Lemma}
\newtheorem{Theorem}{Theorem}
\newtheorem{Corollary}{Corollary}



\begin{document}
\maketitle

\begin{abstract}
We examine the linear completeness of trajectories of eigenfunctions associated to non-linear eigenvalue problems, subject to Dirichlet boundary conditions on a segment. We pursue two specific goals. On the one hand, we establish that linear completeness persists for the non-linear Schr{\"o}dinger equation, even when the trajectories lie far from those of the linear equation where bifurcations occur. On the other hand, we show that this is also the case for a fully non-linear version of this equation which is naturally associated with Appell hypergeometric functions.  Both models shed new light on a framework for completeness in the non-linear setting, considered by L.E.~Fraenkel over 40 years ago, that may have significant potential but which does not seem to have received much attention. 
\end{abstract}

\newpage

\section{Introduction}
In two consecutive papers \cite{F1980a,F1980b} published at the end of the 1970s which appear to have been largely overlooked, L.E.~Fraenkel considered the questions of linear and non-linear completeness for a family of trajectories on a Hilbert space. Taking as one of two models for his investigations\footnote{The other model was the eigenvalue problem with a plus rather than a minus sign in front of the non-linearity.} the semi-linear eigenvalue problem
\begin{equation} \label{nl-schrod}
\begin{aligned} 
&u''-u^3+\lambda u=0 \\
&u(0)=u(1)=0,
\end{aligned}
\end{equation}
he established among several other remarkable results, that a collection of eigenfunctions $\{u_n\}_{n=1}^\infty\equiv \{u_n\}$ of \eqref{nl-schrod} is linearly complete in $L^2(0,1)\equiv L^2$, when subject to a control on the growth of the norms of the $u_n$.  

In terms of the Jacobi elliptic functions, $\sn(y,\mu)$, this family is given explicitly by 
\begin{equation} \label{nl-schrod_efu}
u(x)=u_n(x,\mu)=2^{\frac32}n \mu K(\mu) \sn(2K(\mu)nx,\mu) 
\end{equation} with associated eigenvalues \[ \lambda=\lambda_n=4n^2(1+\mu^2)K(\mu)^2.\] 
Here the modulus $\mu$ lies in $(0,1)$ and it is a free parameter, while the 1/4 period $K(\mu)$ is the complete elliptic integral. According to \cite[Theorem~3.3(i)]{F1980a}, if $\{\mu_n\}_{n=1}^\infty\subset (0,1)$  is a sequence for which 
\begin{equation} \label{growthEF}
    \{\|u_n(\cdot,\mu_n)\|/n\}\in \ell^4(\mathbb{N}),
\end{equation}
then $\{u_n\}$ is a basis of $L^2$. 

A key component in the proof of this result, is the fact that \eqref{growthEF} is equivalent to the condition
\[
      \sum_{n=1}^\infty \left\|\frac{u_n(\cdot,\mu_n)}{n \gamma_n}-e_n\right\|^2<\infty
\]
for $e_n(x)=2^{\frac12} \sin(n\pi x)$ and $\gamma_n$ the $n$-th sine Fourier coefficient of $u_n/n$. 
This is known to be sufficient, but not necessary, for $\{u_n\}$ to become a basis of $L^2$. Moreover, although it allows $\|u_n(\cdot,\mu_n)\|$ to be $O(n^c)$ as $n\to\infty$ for all $c<3/4$, \eqref{growthEF} holds if and only if $\{\mu_n\}\in \ell^4(\mathbb{N})$. Therefore, it requires $\mu\to 0$. In this regime, the equation \eqref{nl-schrod} bifurcates from the linear eigenvalue equation and it is natural to expect that $\{u_n\}$ is close enough to an orthonormal basis of $L^2$. 
   
The present paper is devoted to two specific goals in the context of Fraenkel's original idea of asking questions about linear and non-linear completeness for trajectories of eigenfunctions. On the one hand, we show that  $\{u_n\}$ is also a basis for any $\{\mu_n\}\subset (0,1)$ such that $\limsup \mu_n \leq \mu_0$, where $\mu_0<1$ is a constant very close to 1. Concretely, we show that $\{u_n/n\}$ is a Riesz basis in the sense that it is equivalent to the orthonormal basis $\{e_n\}$. On the other hand, we examine linear completeness on a fully non-linear version of \eqref{nl-schrod}. The eigenvalue equation {\small \begin{equation} \label{p-nl-schrod}
\begin{aligned} 
&(\sgn(\phi')|\phi'|^{p-1})'-(p-1)\sgn(\phi)|\phi|^{2p-1}+\lambda (p-1)\sgn( \phi)|\phi|^{p-1}=0 \\
&\phi(0)=\phi(1)=0,
\end{aligned}
\end{equation}}for $p>1$, in which we have replaced the Laplacian term by the non-linear $p$-Laplacian. 

This new equation is neither artificial nor it is a purposeless generalisation of \eqref{nl-schrod}. Rather, it provides a natural link between the framework of the papers \cite{F1980a,F1980b} and the one developed recently for the $p$-Laplacian and other families of dilated periodic functions arising in the context of Sobolev embeddings on the segment. A systematic treatment of the latter can be found in  the book \cite{EL2011} and the state-of-the-art on the basis question for the $p$-Laplacian and other related non-linear operators can be found in the papers \cite{BE2012} and \cite{BL2015}. For a full list of updated references on the subject and a more general perspective, see also the paper \cite{BM2018}. 

We thank Domenic Petzinna for his useful comments and attentive reading of an earlier version of this manuscript, and Shingo Takeuchi for telling us about the work \cite{ST2012,ST2014}. The background to the concepts and terminology employed below can be found in  the books \cite{H2011} and \cite{WW1920}.

\section{The eigenvalue equation}

Let $p>1$. We examine families of solutions to \eqref{p-nl-schrod}. The case $p=2$ corresponds to \eqref{nl-schrod}. Our interest is completeness properties near and far from bifurcation phenomena associated to the $p$-Laplacian eigenvalue equation 
\begin{equation} \label{p-schrod}
\begin{aligned} 
&(\sgn(\phi')|\phi'|^{p-1})'+\lambda (p-1)\sgn( \phi)|\phi|^{p-1}=0 \\
&\phi(0)=\phi(1)=0.
\end{aligned}
\end{equation}
We know that the eigenfunctions of \eqref{p-schrod} become a Riesz basis of $L^2$ for all $p>p_0$, where $p_0>1$ is a constant close to 1. Specific characterisations of $p_0$ can be found in \cite[Theorem~4.5]{BE2012} and \cite[Theorem~6.5]{BL2015}.

Solutions to \eqref{p-nl-schrod} are given in terms of inverse Appell hypergeometric functions, as follows. For (general) modulus $\mu\in(0,1)$, let 
\[
    K_p(\mu)=\int_0^1 (1-s^p)^{-\frac{1}{p}}(1-\mu^p s^p)^{-\frac{1}{p}}\, \mathrm{d}s
\]
and let $\sn_p(y,\mu)$ be the odd $4K_p(\mu)$-periodic continuous extension of the inverse function of
\[
      w_p(z)=\int_{0}^z (1-s^p)^{-\frac{1}{p}}(1-\mu^p s^p)^{-\frac{1}{p}}\, \mathrm{d}s.
\]
Then, $\sn_p(y,\mu)$ is positive and increasing for $y\in\big(0,K_p(\mu)\big)$. The re-scaled function $\sn_p(2K_p(\mu)x,\mu)$ is 2-periodic, odd and differentiable. It is positive on $(0,1)$ with maximum value equal to 1 at $x=\frac12$. It is also even with respect to that point. The notation we employ here is consistent with that of the paper \cite{ST2012}. 

A full description of the eigenvalues and eigenfunctions for non-linear equations such as \eqref{p-nl-schrod}, and an analysis of the bifurcation phenomenona around those of \eqref{p-schrod}, was established in the paper \cite{GV1988}. See also \cite{ST2012}. If $1<p\leq 2$, a full set of eigenfunctions of \eqref{p-nl-schrod} is given by dilations of $\sn_p(2K_p(\mu)x,\mu)$, for suitable $\lambda>0$. By contrast, for $p>2$, a full set of eigenfunctions for large $\lambda$ also includes the possibility of sub-segments of $(0,1)$ where $\phi(x)$ is  constant. The crucial distinction between the two cases is the fact that $\lim_{\mu\to 1}K_p(\mu)= \infty$ if and only if $1<p\leq 2$. 

In the present manuscript we will only consider trajectories of solutions which are not locally constant, as described in the following theorem. 

\begin{Theorem} \label{Theorem1}
Let $p>1$. A differentiable function $\phi:\mathbb{R}\longrightarrow \mathbb{R}$, such that $\phi$ is not constant on any open segment, satisfies the equation \eqref{p-nl-schrod} for some $\lambda>0$ if and only if the following holds true. For a unique pair $(\mu,n)\in (0,1)\times \mathbb{N}$,
\[\phi(x)=\pm 2^{\frac{p+1}{p}}n \mu K_p(\mu) \sn_p(2nK_p(\mu)x,\mu) \] and \[ \lambda=2^{p}n^p(1+\mu^p)K_p(\mu)^p.\]
\end{Theorem}

Here $n$ determines the number of zeros of the eigenfunction in the segment $(0,1)$ and, alongside with $\mu$ and $p$, it determines its norm. This stametent is a direct consequence of \cite[Theorem~2.1]{GV1988} for $1<p\leq 2$ and \cite[Theorem~2.2]{GV1988} for $p>2$. We give details of the proof. 
 
\begin{proof} We focus on the ``only if'' part of the proof. The ``if'' part may be checked by direct substitution, but it can be better confirmed by reversing the arguments below. 

Let $\phi(x)$ and $\lambda>0$, be such that \eqref{p-nl-schrod} holds and $\phi(x)>0$ for all $x\in (0,1)$.  Then, 
\[
     \begin{aligned}
&[(\phi')^{p-1}]'-(p-1)\phi^{2p-1}+\lambda (p-1) \phi^{p-1}=0, \\
&\phi(0)=\phi(1)=0.
\end{aligned}
\] 
Multiplying each term of this equation by $\frac{p}{p-1}\phi'$, yields
\[
     \frac{p}{p-1}\phi'[(\phi')^{p-1}]'=[(\phi')^{p}]', \qquad \frac{p}{p-1}\phi'(p-1)\phi^{2p-1}=\frac12 (\phi^{2p})' 
\]
and
\[
     \frac{p}{p-1}\phi'\lambda (p-1)\phi^{p-1}=\lambda (\phi^p)'.
\]
Then,
\[
    (\phi')^p=\frac12 \phi^{2p} -\lambda \phi^p+c^p=\frac12 (\alpha-\phi^p)(\beta-\phi^p)
\]
for $c=\phi'(0)>0$ and 
\[
     \alpha,\,\beta=\lambda \pm \sqrt{\lambda^2-2c^p}.
\]
Here $\alpha$ picks the `$+$' sign and $\beta$ the `$-$' sign.
Hence,
\begin{equation} \label{ch-equ-int}
    \phi'=2^{-\frac{1}{p}} \left[(\alpha-\phi^p)(\beta-\phi^p)\right]^{\frac{1}{p}}.
\end{equation}

Since $\phi(0)=\phi(1)=0$, there exists $x_0\in(0,1)$ such that $\phi'(x_0)=0$. As $\phi$ is positive by our assumption, it is then increasing from $x=0$. That is,  $\phi'(x)>0$ for all $x\in [0,x_0)$ and $0<\beta \leq \alpha$ (both roots should be positive and real). Also,
\[
     \phi^p(x_0)=\beta >0 \qquad \text{and} \qquad \lambda^2\geq 2c^p >0.
\]  
Let $a^p=\alpha$ and $b^p=\beta$. Integrating \eqref{ch-equ-int}, gives
\[
       x=\int_{0}^\phi (a^p-t^p)^{-\frac{1}{p}}(b^p-t^p)^{-\frac{1}{p}}\,\mathrm{d}t
\]
for $0\leq x(\phi)\leq x_0$ and $0\leq \phi(x)\leq b$, both increasing. Writing
\[
       x=\frac{2^{\frac{1}{p}}}{ab}\int_0^\phi \left(1-\left(\frac{t}{a}\right)^{p}\right)^{-\frac{1}{p}}
\left(1-\left(\frac{t}{b}\right)^{p}\right)^{-\frac{1}{p}}\,\mathrm{d}t
\] 
and changing variables to $s=\frac{t}{b}$, then calling $\mu=\frac{b}{a}$, gives
\[
      x=\frac{2^{\frac{1}{p}}}{a}
\int_{0}^{\frac{\phi}{b}} (1-\mu^ps^p)^{-\frac{1}{p}} (1-s^p)^{-\frac{1}{p}}\,\mathrm{d}s.
\]
Hence, with $w_p(z)$ as above,
\[
     w_p\left(\frac{\phi(x)}{b}   \right)=\frac{ax}{2^{\frac{1}{p}}}.
\]
The function $\phi(x)$ attains its maximum at $x=x_0$ and $\phi(x_0)=b$.
With the definition of $K_p(\mu)$ as above,
this gives
\[
      K_p(\mu)=\frac{ax_0}{2^{\frac{1}{p}}}.
\]

Now, by virtue of \cite[Theorem~2.1]{GV1988} for $1<p\leq 2$ or \cite[Theorem~2.2]{GV1988} for $p>2$ where for the latter we use the hypothesis that $\phi$ is not constant on any open segment, it follows that $\phi$ is even with respect to $x_0$. Then, $\phi(2x_0)=0$. But, because $2x_0$ is the first zero of $\phi$, then necessarily $x_0=\frac12$. Thus
\[
     a=2^{\frac{p+1}{p}}K_p(\mu).
\]
Therefore, recalling that $b=a\mu$, we have
\[
    \phi(x)=b\sn_p\left(\frac{ax}{2^{\frac{1}{p}}},\mu\right)=
2^{\frac{p+1}{p}}K_p(\mu)\mu \sn_p\left(2K_p(\mu)x,\mu\right)
\]
where $\sn_p(y,\mu)$ is the inverse function of $w_p(z)$ as defined above. 

This determines the expression for a positive eigenfunction $\phi(x)$ as stated in the theorem (that is for $n=1$) and it also shows that this eigenfunction is unique. Moreover, since
\[
     \lambda=\frac{a^p+b^p}{2}=2^p (1+\mu^p)K_p(\mu)^p,
\]
we also obtain the expression for the eigenvalue. So, we know that the eigenpair $(\phi,\lambda)$ is determined uniquely from the pair $(\mu,1)\in(0,1)\times \mathbb{N}$, when $\phi$ is positive.

The claimed statements for $n\geq 2$ follow by evaluating the equation at $n\phi(nx)$. The solution is unique, assuming $\phi'(0)>0$ from \cite[Theorem~2.1]{GV1988} for $1<p\leq 2$, and from \cite[Theorem~2.2]{GV1988} for $p>2$ additionally assuming that $\phi$ is not locally constant. The two signs choice in the conclusion is a consequence of the fact that if $\phi(x)$ is an eigenfunction, then also $-\phi(x)$ is.
\end{proof}

Note that
\begin{equation} \label{representationKp}
\begin{aligned}     K_p(\mu)&=\int_0^1 \frac{x^{\frac{1}{p}-1}}{p}(1-x)^{-\frac{1}{p}}(1-\mu^p x)^{-\frac{1}{p}} \,\mathrm{d}x\\&= \frac{\operatorname{B}\left(\frac{1}{p},\frac{1}{p'}\right)}{p} \ _2\!\operatorname{F}_1\left(\frac{1}{p},\frac{1}{p};1,\mu^p    \right)\end{aligned}
\end{equation}
and
\[
    w_p(z)=zF_1\left(\frac{1}{p},\frac{1}{p},\frac{1}{p},\frac{p+1}{p};z^p,\mu^pz^p    \right).
\]
Therefore, $w_p(z)$ is given in terms of Appell hypergeometric functions \cite{E1950}. 

Turning back to the function $\sn_p(y,\mu)$, let us describe some of its structural  properties. Since
\[
    \frac{\mathrm{d}(\sn_p(y,\mu))}{\mathrm{d}y}=(1-[\sn_p(y,\mu)]^p)^{\frac{1}{p}} (1-\mu^p [\sn_p(y,\mu)]^p)^{\frac{1}{p}}
\]
for all $y\in[0,K_p(\mu)]$, then $\sn_p(y,\mu)$ and its periodic extension are continuously differentiable on $\mathbb{R}$. Moreover, 
\[
     \frac{\mathrm{d^2}(\sn_p(y,\mu))}{\mathrm{d}y^2}=h_p(\sn_p(y,\mu))
\]
where
\[
       h_p(z)=z^{p-1}(1-z^p)^{\frac{2}{p}-1}(1-\mu^pz^p)^{\frac{2}{p}-1}
       ((\mu^p+1)z^p-2)<0
\]
for all $z\in(0,1)$. Then, for the periodic extension (and $p\not=2$), $\sn_p''(y,\mu)<0$ for $y\in[0,K_p(\mu))$ and $\sn_p''(y,\mu)>0$ whenever $y\in (K_p(\mu),2K_p(\mu)]$. At the 1/4-period $y=K_p(\mu)$ and at its odd integer multiples, the second derivative is continuous for all $1<p\leq 2$ but it has a singularity for all $p>2$. Nonetheless, this second derivative is always locally $L^1$, because 
\[
     \int_{K_p(\mu)-\epsilon}^{K_p(\mu)+\epsilon} \!|\sn_p''(y,\mu)|\,\mathrm{d}y=-\sn_p'(y,\mu)\Big|_{y=K_p(\mu)-\epsilon}^{y=K_p(\mu)}\!\!\!+ \sn_p'(y,\mu)\Big|_{y=K_p(\mu)}^{y=K_p(\mu)+\epsilon}\!\!\!<\infty.
\]
That is,
\begin{equation}\label{diffdiff_in_L1loc}
\sn_p''(\cdot,\mu)\in \begin{cases} C(\mathbb{R}) & 1<p\leq 2 \\
L^1_{\mathrm{loc}}(\mathbb{R}) & p>2.\end{cases}\end{equation}
All these basic properties will be invoked in the proof of our main theorem below. 

The next lemma can be regarded as a version of the classical Jordan Inequality. The case corresponding to $\mu=0$ was established in \cite[Proposition~2.3]{BE2012}. 

\begin{Lemma} \label{ineq_snp_lin}
For all $y\in\left(0,K_p(\mu)\right)$,
\[
     \frac{1}{K_p(\mu)} \leq \frac{\sn_p(y,\mu)}{y} \leq 1.
\]
\end{Lemma}
\begin{proof}
    Changing variables to $s=zr$ in the integral defining $w_p(z)$, yields
\[
    w_p(z)=z\int_0^1 (1-y^pr^p)^{-\frac{1}{p}}(1-\mu^pz^pr^p)^{-\frac{1}{p}}\,\mathrm{d}r.
\]
Then, substituting $y=w_p(z)$, gives
\[
     y=\sn_p(y,\mu)\int_0^1 (1-\sn_p^p(y,\mu)r^p)^{-\frac{1}{p}}(1-\mu^p\sn_p^p(y,\mu)r^p)^{-\frac{1}{p}} \, \mathrm{d}r.
\]
Denote the integral on the right hand side by $A$. Since $0<\sn_p(y,\mu)<1$, we conclude that 
\[
     1\leq A\leq \int_{0}^1(1-r^p)^{-\frac{1}{p}}(1-\mu^p r^p)^{-\frac{1}{p}} \, \mathrm{d}r=K_p(\mu).
\]
\end{proof}

For $\underline{\mu}=\{\mu_n\}_{n=1}^\infty\in (0,1)^{\infty}$, we write
\[
    f_n(x)\equiv f_{n,\underline{\mu}}(x)=\sn_p(2K_p(\mu_n)nx,\mu_n).
\]
Whenever it is sufficiently clear from the context, we leave implicit the dependence of $\{f_n\}$ on $\underline{\mu}$. Then,
\[
     \phi_n(x)=2^{\frac{p+1}{p}}\mu_n nK_p(\mu_n) f_n(x) 
\]
form a collection of eigenfunctions of \eqref{p-nl-schrod}, for $n\in\mathbb{N}$. Our next statement is the first main contribution of this paper. It gives sufficient conditions on $\underline{\mu}$ for $\{f_{n,\underline{\mu}}\}$ to be a Riesz basis of $L^2(0,1)$. 

\begin{figure}[t]
\includegraphics[width=60mm]{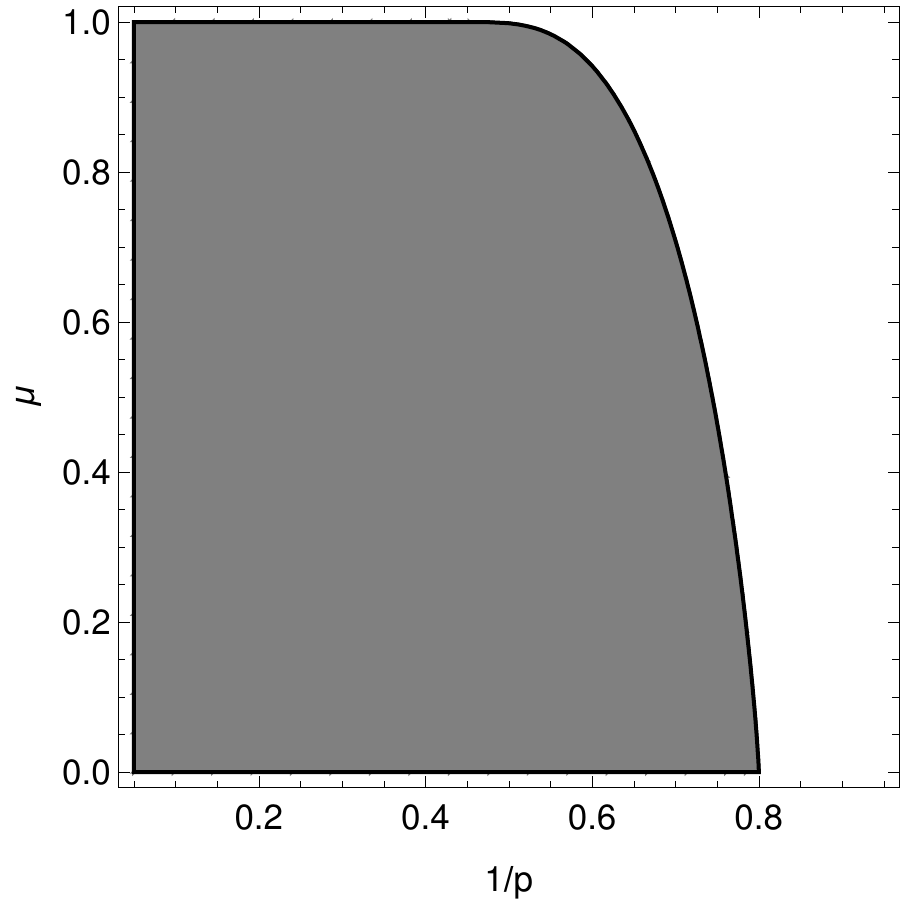}
\caption{\label{Figure1}In the shaded region, $K_{p}(\mu)<\frac{8}{\pi^2-8}$. The numerical approximations employed might not be too accurate for $\mu$ near 1. But away from that region, the graph gives a sharp approximation illustrating the interplay between the two parameters for the conclusion of Theorem~\ref{fn_basis} to hold true.}
\end{figure}

\begin{Theorem} \label{fn_basis}
 Let $p>1$ and $\underline{\mu} \in (0,1)^{\infty}$. If
\begin{equation} \label{firstcond}
      \sup_{n\in\mathbb{N}}K_p(\mu_n)<\frac{8}{\pi^2-8},
\end{equation}
then $\{f_{n,\underline{\mu}}\}$ is a Riesz basis of $L^2(0,1)$.
\end{Theorem}
\begin{proof}
We split the argument into five different steps. 

\underline{Step~1}. We describe explicitly the linear operator $A:e_n\longmapsto f_{n,\underline{\mu}}$ in terms of isometries and diagonal operators of $L^2$. 

Let
\[
    \tau_{k}(n)=\sqrt{2}\int_0^1 \sn_p(2K_p(\mu_n)x,\mu_n) \sin(k\pi x)\,\mathrm{d}x
\]
be the $k$-th sine Fourier coefficient of the function $\sn_p(2K_p(\mu_n)\cdot,\mu_n)$. Since the latter is even with respect to $x=\frac12$, then $\tau_j(n)=0$ for all even index $j$ and all $n\in\mathbb{N}$. Let isommetries $M_k:L^2\longrightarrow L^2$ be given by $M_{k}e_n(x)=e_{kn}(x)$. Let the diagonal operators $A_k:L^2\longrightarrow L^2$ be given by $A_k=\operatorname{diag}[\tau_k(n):n\in\mathbb{N}]$. That is, we fix the index $k$ of the Fourier coefficient and move the index $n$ of the entries of $\underline{\mu}$.

Since
\[
     f_{n,\underline{\mu}}(x)=\sum_{k=1}^\infty \tau_k(n) M_k e_n(x)=
 \sum_{k=1}^\infty  M_k A_k e_n(x),
\]
and both families of operators are bounded, then
\[
     A=\sum_{k=1}^\infty M_kA_k.
\]
Here the series is absolutely covergent in operator norm. This follows from the arguments given in Step~4 below, as these arguments show that
\[
    \sum_{k=1}^\infty \|M_k A_k\|=\sum_{k=1}^\infty \|A_k\|<\infty.
\]

\underline{Step~2}. We claim that
\[
       \tau_1(n)\geq \frac{4\sqrt{2}}{\pi^2}
\]
for all  $n\in\mathbb{N}$, irrespective of the choice of $\underline{\mu}$. Indeed, by substituting $y=2xK_p(\mu_n)$ in Lemma~\ref{ineq_snp_lin}, follows that
\begin{align*}
   \tau_1(n)& = 2\sqrt{2}\int_0^{\frac12} \sn_p(2 K_p(\mu_n)x,\mu_n) \sin(\pi x) \,\mathrm{d}x \\
&\geq 4\sqrt{2}  \int_0^{\frac12} x  \sin(\pi x)\,\mathrm{d}x =\frac{4\sqrt{2}}{\pi^2}.
\end{align*}

\underline{Step~3}. We next show that
\[
    \sum_{\substack{k=3\\ k\equiv_2 1}} \sup_{n\in \mathbb{N}}|\tau_k(n)| \leq \frac{4 \sqrt{2}}{\pi^2} \left(\frac{\pi^2}{8}-1 \right)\sup_{n\in \mathbb{N}}K_p(\mu_n).
\]
For this purpose, let $g(x)=\sn_p(2K_p(\mu_n)x,\mu_n)$ and recall the regularity properties of $\sn_p(y,\mu)$ given in \eqref{diffdiff_in_L1loc}. For $k\geq 3$ odd, integrating by parts twice and noting that $g'(\frac12)=0$, gives 
\[
    \tau_k(n)=\frac{2 \sqrt{2}}{n^2 \pi^2} \int_0^{\frac12} g''(x) \sin(k\pi x)\,\mathrm{d}x.
\] 
Hence, since $g''(x)<0$ for all $x\in (0,\frac12)$, 
\begin{align*}
|\tau_k(n)|&\leq  \frac{2 \sqrt{2}}{k^2 \pi^2} \int_0^{\frac12} |g''(x)| \,\mathrm{d}x =
\frac{2 \sqrt{2}}{k^2 \pi^2} \left(g'(0)-g'\Big(\frac12\Big)\right) \\
&= \frac{4 \sqrt{2}K_p(\mu_n)}{k^2 \pi^2}
[\sn_p'(0,\mu_n)-\sn_p'(K_p(\mu_n),\mu_n)]=
\frac{4 \sqrt{2}K_p(\mu_n)}{k^2 \pi^2}. 
\end{align*}
By taking the suprema in $n$ and then the summation in the index $k$, this yields the claim made above.

\underline{Step~4}. If
\begin{equation} \label{hypo_invert}
     \sum_{\substack{k=3\\ k\equiv_2 1}} \sup_{n\in \mathbb{N}} |\tau_k(n)|<\inf_{n\in \mathbb{N}}\tau_1(n),
\end{equation}
then $A$ is an invertible operator. Note that the right hand side of this inequality is always positive according to the step~2. 

Assume that \eqref{hypo_invert} holds true. To show that $A$ is invertible, firstly note that the left hand side of this inequality equals 
\[
      \sum_{\substack{k=3\\ k\equiv_2 1}} \|M_kA_k\|.
\]
Indeed, 
\[
    \|A_k\|=\sup_{n\in \mathbb{N}} |\tau_k(n)|. 
\]
And, since $M_k$ are isommetries, for all $\varepsilon>0$ there is $\tilde{v}\in L^2$, such that $\|\tilde{v}\|=1$ and
\[
    \|M_kA_k\tilde{v}\|=\|A_k\tilde{v}\|> \|A_k\|-\varepsilon.
\]
Then $\|A_k\|\geq \|M_kA_k\|\geq \|A_k\|-\varepsilon$. Taking $\varepsilon\to 0$ gives 
\begin{equation}  \label{normlhs}
    \sum_{\substack{k=3\\ k\equiv_2 1}} \sup_{n\in \mathbb{N}} |\tau_k(n)|=
\sum_{\substack{k=3\\ k\equiv_2 1}}^\infty \|M_kA_k\|.
\end{equation}

Now, since $\inf_{n\in \mathbb{N}}\tau_1(n)>0$ and
\[
     \|A_1^{-1}\|= \frac{1}{\inf_{n\in \mathbb{N}}\tau_1(n)},
\]
then $A_1$ is invertible. Note also that $M_1=I$, the identity operator. Then, 
\[
   A=M_1A_1+\sum_{\substack{k=3\\ k\equiv_2 1}}^\infty M_kA_k=
   A_1\Big(I+A_1^{-1}\sum_{\substack{k=3\\ k\equiv_2 1}}^\infty M_kA_k\Big).
\]
Moreover, from \eqref{normlhs} and the hypothesis \eqref{hypo_invert}, we have
\[
    \Big\|A_1^{-1}\sum_{\substack{k=3\\ k\equiv_2 1}}^\infty M_kA_k\Big\|\leq \|A_1^{-1}\|\sum_{\substack{k=3\\ k\equiv_2 1}}^\infty\|M_kA_k\|<1.
\]
Hence, indeed $A$ is invertible. 

\underline{Step~5}. According to step~3, \eqref{firstcond} implies that  
\[
    \sum_{\substack{k=3\\ k\equiv_2 1}} \sup_{n\in \mathbb{N}} |\tau_k(n)|<\frac{4\sqrt{2}}{\pi^2}.
\]
But from step~2, we know that \eqref{hypo_invert} holds true. Therefore, as $A$ is invertible and $A:e_{n}\longmapsto f_n$, we have $\{f_n\}$ equivalent to the orthonormal basis $\{e_n\}$.
\end{proof}

Note that the condition \eqref{firstcond} holds for $p\approx 2$ and $\sup \mu_n$ small enough, not necessarily approaching zero as $n\to\infty$. Indeed $K_p(\mu)\approx \pi/2$ under these conditions. Figure~\ref{Figure1} shows the interplay between the parameters $\mu$ and $1/p$ for the hypothesis of the theorem to be verified. The value of $K_p(\mu)$ there was found from a computer approximation of the representation \eqref{representationKp}.

The proof of the next statement follows in a straightforward manner from Theorem~\ref{fn_basis} by  changing a finite number of terms in the sequence $\mu_n$ and re-scaling $f_{n,\underline{\mu}}$.  

\begin{Corollary} \label{corollary1}
If $\{\mu_n\}_{n=1}^\infty\subset (0,1)$ is such that 
\[
    \limsup_{n\to \infty} K_p(\mu_n)<\frac{8}{\pi^2-8},    
\] 
then the family $\{
\phi_n(\cdot,\mu_n)\}$ of eigenfunctions of \eqref{p-nl-schrod} forms a basis of $L^2(0,1)$.
\end{Corollary}

The statements of Theorem~\ref{fn_basis} and Corollary~\ref{corollary1} are in line with the findings of \cite{ST2014}. The latter work examined bases properties of a different, but related, family of periodic functions in a regime where the corresponding generalised modulus sequence $\{\mu_n\}$ is constant. In particular, the condition given in \cite[Theorem~6.1]{ST2014} is consistent with the condition we found above. We will comment on the specific case $p=2$ at the end of the the next section.

%%%%%%%%%%%%%%%%%%%%%%%%%%%%%%%%%

\section{The semi-linear case}

We now focus on the case $p=2$. Lemma~3.2(i) and Theorem~3.3(i) of \cite{F1980a}, establish conditions for the eigenfunctions of \eqref{nl-schrod} to be a basis of $L^2$ by means of a different criterion than the one invoked in the proof of Theorem~\ref{fn_basis} above. These conditions are given in terms of a parameter, 
\[
     s=\pm\frac{4\pi q^{\frac12}}{1-q}\in\mathbb{R}
\]
where $q$ is the nome and the sign convention matches that of the eigenfunction. See \cite[(2.9)]{F1980a}.
Concretely, for the sequence $s=r_n$,  we know that
\begin{equation}  \label{eq3}
      \sum_{n=1}^\infty \left\|\frac{u_n}{n \langle u_n,e_n\rangle}-e_n\right\|^2<\infty
\end{equation}
 if and only if $\{r_n\}\in\ell^4$. This, alongside with $\omega$-linear independence, ensures that the family $\{\frac{u_n}{n \langle u_n,e_n\rangle}\}_{n=1}^\infty$ is a Riesz basis of $L^2$. The latter can, for example, be derived directly from the proof of \cite[Theorem~2.20, p.265]{K1980}. 

In \cite{F1980a}, the choice of the alternative parameter $s$ was convenient so to confirm the validity of \eqref{eq3}. In terms of the modulus $\mu\in(0,1)$,  \cite[p486]{WW1920} we know that  $\mu\sim q^{\frac12}$ as $q\to 0$. Hence, \eqref{eq3} is equivalent to the choice of $\mu=\mu_n$ in each of the bifurcation curves to be $\{\mu_n\}\in \ell^4$ also. As we can deduce from Theorem~\ref{fn_basis}, in the case $p=2$, the latter is sufficient but not necessary, for the family $\big\{\frac{u_n}{n \langle u_n,e_n\rangle}\big\}$ to become a Riesz basis. Our main goal now is to show that this family is in fact a Riesz basis for any choice of $\mu_n\in(0,1)$ such that $\sup \mu_n< \mu_0$ where $\mu_0$ is substantially closer to $1$.

The Jacobi elliptic function has Fourier expansion
\[
    \operatorname{sn}\big(2K(\mu)x\big)=\frac{2\pi q^{\frac12}}{K(\mu)\mu}\sum_{j=0}^\infty \frac{q^j}{1-q^{2j+1}} \sin\big((2j+1)\pi x\big).
\]
Let
\begin{equation} \label{eq4}
     g(x,\mu)\equiv g(x)=(1-q)\sum_{j=0}^\infty \frac{q^j}{1-q^{2j+1}}e_{2j+1}(x).
\end{equation}
Then, the eigenfunctions \eqref{nl-schrod_efu} of \eqref{nl-schrod} are
\[
     u_n(x)=\frac{2^{\frac52}\pi n q^{\frac12}}{1-q}g(nx,\mu).
\]
In order to establish an improvement to the statement of Theorem~\ref{fn_basis} for this specific case, we first characterise the summation of the Fourier coefficients of $g(x)$. 

For $0<\beta<1$, let  the Lambert series
\[
   L(\beta)=\sum_{n=1}^\infty \frac{\beta^n}{1-\beta^n}.
\]
Consider the $q$-digamma function
\[
     \psi_q(x)=\frac{\mathrm{d}}{\mathrm{d}x}\log \Gamma_q(x)=\frac{\Gamma_q'(x)}{\Gamma_q(x)}
\]
where 
\[
    \Gamma_q(x)=(1-q)^{1-x}\prod_{n=0}^{\infty} \frac{1-q^{n+1}}{1-q^{n+x}}
\]
is the $q$-gamma function. Then 
\[
     \psi_q(x)=-\log(1-q)+(\log q)\sum_{n=1}^\infty \frac{q^{nx}}{1-q^n},
\]
for all $x>0$ and $q\in(0,1)$. See \cite[(1.5)]{AG2007}. Hence,
\[
      L(\beta)=\frac{\psi_\beta(1)+\log(1-\beta)}{\log(\beta)}.
\]

We claim that 
\[ 
    \sum_{n=0}^\infty \frac{\beta^{2n+1}}{1-\beta^{4n+2}}=L(\beta)-2L(\beta^2)+L(\beta^4).
\]
Indeed, 
\[
    C=L(\beta)-L(\beta^2)=\sum_{n=0}^\infty \frac{\beta^{4n+1}}{1-\beta^{4n+1}}+\frac{\beta^{4n+3}}{1-\beta^{4n+3}}
\]
and
\[
    D=L(\beta^4)-L(\beta^2)=-\sum_{n=0}^\infty \frac{\beta^{4n+2}}{1-\beta^{4n+2}}
\]
add up to
\[
   C+D=\sum_{n=0}^\infty \frac{(1+\beta^{2n+1})\beta^{2n+1}-\beta^{4n+1}}{1-\beta^{4n+2}}
\]
which is equal to the left hand side of the above claim. From it, we then gather that
\begin{equation}   \label{SumLambert}
     \sum_{n=1}^\infty \frac{q^n}{1-q^{2n+1}}=\frac{L(\sqrt{q})-2L(q)+L(q^2)}{\sqrt{q}}-\frac{1}{1-q}.
\end{equation}
This identity will now be crucial in the proof of the next theorem. The latter is the other main contribution of this paper.

\begin{Theorem} \label{casep2}
Let $q_0\in(0,1)$ be such that
\begin{equation} \label{sharp}
      \frac{L(\sqrt{q_0})-2L(q_0)+L(q_0^2)}{\sqrt{q_0}}=\frac{2}{1-q_0}
\end{equation}
and let
\[
    \mu_0=\frac{\vartheta_2^2(0,q_0)}{\vartheta_3^2(0,q_0)}.
\] 
Let $\{\mu_n\}_{n=1}^\infty\subset (0,1)$. If $\sup \mu_n< \mu_0$, then $\{g(n\cdot,\mu_n)\}_{n=1}^\infty$ is a Riesz basis of $L^2$.
\end{Theorem}
\begin{proof}
We proceed as in the proof of Theorem~\ref{fn_basis}. Let
\[
    \rho_k(n)=\langle g(\cdot,\mu_n),e_k\rangle.
\]  
That is, the $j$-th Fourier coefficient of the function $g(x,\mu_n)$. Let
\[
     B_k=\operatorname{diag}[\rho_k(n)\,:\,n\in\mathbb{N}].
\]
Then, carrying over the notation from the step~1 of the proof of Theorem~\ref{fn_basis}, we have that
$
       g(nx,\mu_n)=Be_n(x)
$
for all $n\in\mathbb{N}$,
where
\[
     B=\sum_{k=1}^\infty M_kB_k.
\]
According to \eqref{eq4}, $\rho_1(\mu_n)=1$ for all $n\in \mathbb{N}$. Then,
$B_1=I$. The proof of the present theorem reduces to showing that $B$ is an invertible bounded operator acting on $L^2$.

Arguing as in step~4 of the proof of Theorem~\ref{fn_basis}, if
\begin{equation} \label{trianglecasep2}
      \sum_{\substack{k=3 \\ k\equiv_2 1}} \sup_{n\in\mathbb{N}} \rho_k(n)<1,
\end{equation}
then $B$ is invertible. We now confirm this inequality.

Let $\mu_0$ be as in the hypothesis. Then \cite[p.486]{WW1920}, the nome associated to $\mu_0$ is $q_0$ satisfying \eqref{sharp}. For each fixed $j\in \mathbb{N}$, by differentiating with respect to $q$, it is straightforward to see that the function
\begin{equation} \label{monoto}
      q\longmapsto \frac{(1-q)q^j}{1-q^{2j+1}}=\frac{1}{q^{-j}+\cdots+q^{-1}+1+q^1+\cdots+q^j}
\end{equation}
is increasing as $q$ increases.  Then,
\[
     \sup_{n\in\mathbb{N}} \rho_{2j+1}(n)< \frac{(1-q_0)q_0^j}{1-q_0^{2j+1}},
\]
for all $j\in \mathbb{N}$. Now, let
\[
     S=\sum_{j=1}^\infty \frac{(1-q_0)q_0^j}{1-q_0^{2j+1}}.
\]
According to \eqref{SumLambert} and clearing from \eqref{sharp}, $S=1$. Hence, indeed
 \eqref{trianglecasep2} holds true and the theorem is valid.
\end{proof}

This theorem implies that,  whenever $p=2$, the conclusion of Corollary~\ref{corollary1} holds true for $\{\mu_n\}_{n=1}^\infty\subset(0,1)$ such that $\limsup \mu_n< \mu_0$.  Three comments about this are now in place. 

Firstly, note that the condition \eqref{sharp} is optimal in the following precise sense. Due to the monotonicity in $q$ of the terms \eqref{monoto}, the inequality \eqref{trianglecasep2} reverses for $\mu>\mu_0$ and the argument leading to the invertibility of $B$ is no longer valid.

Secondly, the condition \eqref{firstcond} for $p=2$ holds true, only for $q\in(0,0.315323)$ which corresponds to $\mu\in(0,0.996912)$. By contrast, from numerical estimations of the $q$-digamma function and substitution, \eqref{sharp} holds true for $q_0\approx 0.768062$. This corresponds to
$1-\mu_0<10^{-7}$. Therefore, Theorem~\ref{casep2} significantly improves the general Theorem~\ref{fn_basis} for $p=2$. 

Finally, for constant $\{\mu_n\}$, it was reported in \cite{ST2014} that $\mu_n= 0.9909$ was a valid threshold for basis in the case $p=2$. The current findings confirm this claim also in the case of non-constant $\{\mu_n\}$.    
%%%%%%%%%%%%%%%%%%%%%%%%%%%%%%%%%%%%%%%%%%%%%%%%%%%%%%

\end{document}